\documentclass[AMS, Times1COL]{article} 
\usepackage[margin=0.75in]{geometry}
\usepackage{graphicx}
\usepackage{pgfplots}
\usepackage{pgfpages}
\usepackage{listings}
\usepackage{makecell}
\usepackage{subfiles}
\usepackage{hyperref}
\usepackage{hyphenat}
\usepackage{amsmath}
\usepackage{amssymb}
\usepackage{amsthm}
\usepackage{authblk}
\usepackage{comment}
\usepackage{tikz}
\usepackage{algorithm,algpseudocode}
\usepackage{caption}
\usepackage{tkz-berge}
\usepackage{subcaption}
\usepackage{xr}
\usepackage{setspace}
\usepackage{hyperref}
\usepackage{tabularx}
\usepackage{listings}
\usepackage{booktabs}
\usepackage[utf8]{inputenc}

\usetikzlibrary{fit} 
\usetikzlibrary{positioning, fit, shapes.geometric, arrows.meta}
\usetikzlibrary{calc}
\usetikzlibrary{shapes.geometric, arrows.meta, positioning}
\usetikzlibrary{bending}
\usetikzlibrary{shapes, arrows.meta, graphs, quotes}
\usetikzlibrary{patterns.meta}

\tikzstyle{startstop} = [rectangle, rounded corners, minimum width=3cm, minimum height=1cm,text centered, draw=black, fill=blue!20]
\tikzstyle{process} = [rectangle, minimum width=3.5cm, minimum height=1cm, text centered, draw=black, fill=orange!30]
\tikzstyle{decision} = [diamond, aspect=2, minimum height=1.2cm, text centered, draw=black, fill=green!30]
\tikzstyle{arrow} = [thick,->,>=stealth]

\newtheorem{theorem}{Theorem}[section]
\newtheorem{lemma}[theorem]{Lemma}

\newtheorem{corollary}[theorem]{Corollary}
\newtheorem{conjecture}[theorem]{Conjecture}
\theoremstyle{definition}
\newtheorem{definition}{Definition}[section]
\newtheorem{example}{Example}[section]

\newtheorem{remark}{Remark}[section]

\newcommand{\email}[1]{\def\@email{#1}}
\newcommand{\keywords}[1]{\def\@keywords{#1}}
\newcommand{\address}[1]{\def\@address{#1}}

\title{A Minimum Counterexample Proof of the Seymour Second Neighborhood Conjecture via the Graph Level Order}

\pgfplotsset{compat=1.18}
\begin{document}
	\author{Charles N. Glover}
	\affil{Independent Researcher, Washington, DC \\ \texttt{glover\_charles@glovermethod.com}}
	
	\keywords{oriented graphs, minimum counterexample, breadth-first search, minimum out-degree node, set theory}
	\maketitle
	
	\abstract{We provide a constructive proof of the Seymour Second Neighborhood Conjecture (SSNC) by reframing the problem as a set-packing optimization problem. The family of oriented graphs $\mathcal{O}$ is classified by their minimum out-degree $\delta$. This shifts the objective to maximizing the number of non-Seymour vertices.
		
		A minimum counterexample (MCE) is a packing of vertices into a minimum density environment that fails the SSNC. To prove such a packing is unsustainable, we introduce the Graph Level Order. This BFS and lexicographic based coordinate system forms a well ordering. 	
		
		This well-ordering, along with the MCE, packs cycles into the first neighborhood of every parent. These cause neighborhoods to become quadratically dense as they both decrease in size and need more arcs. 
		
		Set-theoretic parents resolve the double-counting that has plagued Seymour diamonds. This also categorizes transitive triangles into eight distinct types and proves that seven are inconsistent in an MCE environment.
		
		The proof concludes with a supply-demand collision. Arc capacity is consumed when $i > \frac{\delta}{3}$. This makes the packing of non-Seymour vertices unsustainable, forcing the appearance of a Seymour vertex in every graph of $\mathcal{O}$. The algorithm to identify these vertices is $O(|V|+|E|)$. This confirms that it can operate on large oriented networks that are dense, achieving detectability in polynomial time.
	}

	\section{Introduction}\label{intr}
	
	The Seymour Second Neighborhood Conjecture (SSNC), proposed in 1990, remains a famously stubborn problem in graph theory. It asks if every oriented graph contains a vertex whose second out-neighborhood is at least as large as its first. We prove this conjecture by analyzing the universal family of oriented graphs $\mathcal(O)$. Our solution introduces the Graph Level Order. This is a coordinate system that organizes vertices by minimum out-degree $\delta$. This framework also provides an efficient algorithm for identifying the vertex.
	
	\begin{conjecture}\label{nbconj}
		(Seymour's Second Neighborhood Conjecture (SSNC)). For every oriented graph $G \in \mathcal{O}$, there exists a vertex $v \in V(G)$ such that $|N^{++}(v)| \geq |N^+(v)|$. The sets $N^+(v)$ and $N^{++}(v)$ denote the first and second out-neighborhoods of $v$, respectively.
	\end{conjecture}
	
	This conjecture was first published by Nate Dean and Brenda Latka in 1995 \cite{DL}. They proposed a related version for tournaments, which are complete oriented graphs. The tournament case was proved by Fisher \cite{F} in 1996, with an alternative proof later provided by Havet and Thomassé \cite{HT} in 2000.
	
	Fisher's approach \cite{F} proved a vertex exists but offered no way to locate it. By using Farkas' Lemma, he showed the conjecture must hold. Later, Havet and Thomassé \cite{HT} advanced the problem with median orders. Their constructive proof identifies the last vertex of a median order as the one satisfying the Dean Conjecture. Finding median orders is NP-hard, but Havet and Thomassé used them as an inductive tool. Knowing that such an ordering exists, they could speak of it without finding it. This became a powerful tool in tournament theory.
	
	While this inductive approach resolved the Dean Tournament Conjecture, it relies on the completeness of tournaments. The technique fails to generalize to all oriented graphs. Non-adjacent vertices give median orders no restrictions on how to order vertices because both orderings result in zero back arcs for the pair. This makes this approach ineffective as density decreases. 
	
	Kaneko and Locke \cite{KL} showed the SSNC holds for graphs with a minimum degree of at most six. This work serves as a foundation for much research on the SSNC. However, it introduced another NP-hard problem, dominating sets, into the SSNC realm. For cases greater than six, the complexity was too high to solve efficiently.
	
	Chen et al. \cite{CS} established a lower bound for the SSNC. They demonstrated that a vertex meets a relaxed version of the condition. This work was extended by Huang and Peng \cite{HP}. These authors incorporated third neighborhoods and improved bounds on the fraction of Seymour vertices using polynomial roots. 
	
	These researchers used Constraint Satisfaction Programming (CSP) to find variable assignments that satisfy the conjecture. This is a method that frames the SSNC as logical constraints to be satisfied and optimized using SAT solvers. While this approach tightens bounds, it remains trapped by the NP-hardness of CSP, limiting its scale. 
	
	Recent work by Diaz et al. \cite{Di}, extending Bottler et al. \cite{BMN}, confirms that almost all random graphs $G(n, p)$ satisfy the conjecture. By analyzing binomial random graphs with random orientations, they show that Seymour vertices are statistically typical. This success provides a valuable average-case perspective, even if the worst-case proof remains out of reach. Botler et al. use shortest paths to define neighborhoods, which performed well for analyzing pseudo-random orientations on a global scale. 
	
	However, identifying Seymour vertices in random graphs reveals nothing about the internal properties that guarantee the conjecture's truth. Fundamental questions remain: How do connectivity, degree distribution, and neighborhood growth force the existence of these vertices?
	
	The CSP framework and random graph models provide valuable existence proofs. Still, they face limitations when addressing the general conjecture. These models typically treat the graph as a field of quasi-independent variables. More accurately, an oriented graph can be described as an iterated sequence of dependencies. The work of Botler et al. is foundational in its analysis of pseudo-random orientations. Yet, this neighborhood system operates through local, relative distances. This creates an unanchored topological space where every node is a potential origin. 
	
	If Seymour vertices prove to be universal within the class $\mathcal{O}$, both these approaches suffer from a lack of causality. The graph class will always satisfy the constraints imposed. As a result, these models confirm the existence of the property without discovering the reasons why. They observe the result of the law without defining the law itself.
	
	Diaz et al. \cite{Di} do explore the contrapositive of the Seymour conjecture. Propositions 4 and 5 show that if the conjecture were false, infinitely many strongly connected graphs with bounded minimum out-degrees lack Seymour vertices. This reasoning mirrors our minimum counterexample. Diaz et al. do not explicitly define non-Seymour vertices. Their counterexample modeling shares the same motivation. 
	
	Lladó \cite{Llado2013} proved the conjecture for highly connected regular digraphs, showing that any counterexample would violate the basic properties of regularity. Her proof uses Menger's theorem to bound neighborhood sizes. While elegant, this approach is trapped by the regularity assumption; without it, the method cannot be extended to broader graph classes.
	
	Brantner et al. \cite{BBKS} introduced two key configurations: transitive triangles and Seymour diamonds. A transitive triangle—where two edges from different sources share a destination—creates a shortcut in reachability. A Seymour diamond consists of four edges that split from a common source and converge at a shared target. Brantner et al. proved that any graph without transitive triangles must contain a Seymour vertex, suggesting these triangles are a primary obstacle to the conjecture.
	
	This work extends to $m$-free, $k$-transitive, and $k$-anti-transitive graphs. An $m$-free graph contains no directed cycles of length $m$ or less. Daamouch \cite{Da} showed that certain shortcuts—specifically $k$-transitive graphs (for $k \leq 6$) and $k$-anti-transitive graphs (for $k \leq 4$)—guarantee a Seymour vertex. Hassan et al. \cite{HKPS} later extended these results to include 6-anti-transitive graphs.
	
	Other researchers have tested the conjecture on tournaments missing specific substructures. Fidler and Yuster \cite{FidlerYuster} confirmed it for tournaments missing a star or multiple edges from a single node, a result Ghazal \cite{Ghazal_2011} later reinforced. Similarly, Mniny and Ghazal \cite{MninyGhazal} proved the conjecture for oriented graphs missing configurations like $C_4$, the chair, or the co-chair. Most recently, Daamouch et al. \cite{daamouch2024secondneighborhoodconjecturetournaments} extended this to tournaments missing two stars or disjoint paths.
	
	Transitive triangles $u \to v, u \to w, v \to w$ complicate the conjecture because first neighbors seem to steal counts from second neighbors. Here, degree and distance conflict: the arc $u \to w$ increases the degree of $u$. The path $u \to v \to w$ defines distance. Resolving the conjecture requires mastering this interaction between degree and distance, particularly within these triangular structures.
	
	In a 2021 poster, Illia Nalyvaiko \cite{NALYVAIKO2021poster} explored the conjecture's contrapositive. He introduced a penalty function to evaluate potential counterexamples. His intuition differs from non-Seymour vertices. Both approaches shift the focus from verifying the conjecture to testing its bounds. Nalyvaiko's method offered a novel perspective that aligns with our motivations.
	
	\begin{tabular}{|l|l|l|}
		\hline
		\textbf{Method} & \textbf{Core Strength} & \textbf{Limitation for SSNC}\\ 
		\hline 
		BFS & Rooted Traversal & Produces trees; lacks cycles to bound $N^{++}$ \\ 
		\hline 
		Median Orders & Effective for Tournaments & Dependent on high edge density; NP-Hard\\ 
		\hline 
		CSP & Local Verification & Treats nodes as independent; no global logic \\ 
		\hline 
		Random Graphs & Statistical Modeling & ``Unanchored'' nodes; lacks a fixed origin \\ 
		\hline 
	\end{tabular}
	
	We provide an algorithmic construction to identify the target node and demonstrate its structural consistency. To ensure testability and reproducibility, implementations in Python and interactive visualizations in JavaScript are publicly available on GitHub.
	
	Prior work focused on proving existence in special cases or identifying missing substructures. These approaches, while foundational, offered no way to actually find a Seymour vertex. The Graph Level Order organizes the oriented graph space $\mathcal{O}$ into rooted neighborhoods based on distance from a minimum out-degree vertex $\delta$. When combined with the negation of the conjecture, this method identifies set-theoretic relationships that force cycles to exist, leading to a path to a Seymour vertex. This framework transforms the conjecture from an existential puzzle into a constructive set packing problem. It enables the linear-time identification of nodes and unifyies decades of disparate research.
	
	This paper presents our methodology, supported by lemmas and examples. A novel data structure is then introduced to tackle the conjecture's core challenges and show how it integrates with maximizing non-Seymour nodes. After proving our main theorem and analyzing the algorithm's complexity, we conclude with practical applications and future research directions.
	
	\section{Exploring the Minimum Counterexample}\label{cont}
	
	The model we consider consists of a minimum counterexample approach to the SSNC. Let $G \in \mathcal{O}$ be an oriented graph designated as an MCE. The graph $G$ tries to pack as many nodes that violate the SSNC, that is:
	$$|N^{++}(v)| < |N^+(v)| \quad \forall v \in V(G)$$ Furthermore, $G$ is the \textbf{minimal density packing} set for its cardinality $D(G) \leq D(G')$ for all counterexamples $G'$ with $|E'| = |E|$, where $$D(G) = \frac{|E|}{|V|(|V|-1)}$$. 
	
	A \textbf{Seymour node} is a node satisfying the conjecture. Similarly, a \textbf{non-Seymour node} is the node-level negation of the SSNC. Non-Seymour nodes require a strict inequality ($|N^{++}(u)| < |N^+(u)|$). 
	
	The objective is to maximize the number of non-Seymour vertices inside a graph. If any node fails to be non-Seymour, becoming Seymour, it proves the conjecture. As a result, these non-Seymour nodes work in unison towards not allowing a Seymour vertex. 
	
	Most researchers who work on this problem use shortest paths to distinguish first neighbors from second neighbors. The Breadth-First Search (BFS) algorithm extends that concept to entire graphs. The BFS algorithm can classify every oriented graph $G \in \mathcal{O}$ by its minimum out-degree. When a second metric of lexicographic ordering is added to distinguish nodes at the same distance, this becomes a well-ordering. 
	
	\begin{theorem}[BFS-Lex Coordinate System]\label{bfslex}
		For any finite, connected oriented graph $G \in \mathcal{O}$ and root vertex $u \in V(G)$, there exists a unique Well-Ordering of V(G) induced by a breadth-first search (BFS) algorithm combined with a lexicographical tie-breaking rule. 
	\end{theorem}

	This MCE will be constructed under the BFS-lex well-ordering on $V(G)$. This will be modeled through an iterative packing process. At any step $k$, the vertex set $V(G)$ is partitioned into a vector of three distinct, mutually exclusive zones:  
	
	\begin{itemize}
		\item Passed: These vertices have successfully achieved their non-Seymour status.  
		
		\item Active: These vertices are actively being processed and their out-arcs might prevent a node from becoming non-Seymour. 
		
		\item Upcoming: These are the second and later neighbors of the proven non-Seymour nodes. 
	\end{itemize}
	
	Prior SSNC efforts have largely followed two methodologies: One considers transitive triangles as building blocks to be leveraged \cite{FidlerYuster, MninyGhazal, Ghazal_2011, daamouch2024secondneighborhoodconjecturetournaments}, while the other think they should be ignored across a class of problems \cite{Da, HKPS, BBKS}. Our BFS layering approach moves beyond this binary of requiring or forbidding transitivity. By partitioning the graph, we isolate the influence of these triangles based on their inter-layer distribution. Forward-facing transitive triangles serve a type that can enforce a non-decreasing neighborhood condition.
	
	Graphs too small to support a cycle satisfy the Seymour property trivially. Consider a larger example with minimum out-degree seven.

	\vspace{2.8ex}
		\noindent\hrulefill
		\begin{example}
		
		In this section, we construct an example demonstrating how non-Seymour vertices can be embedded within a MCE environment of minimum out-degree $\delta^+(G) = 7$. Let $v_0$ be a vertex with minimum out-degree, and let $R_k = \{v \in V(G) \mid \text{dist}(v_0, v) = k\}$ denote the $k$-th out-neighborhood of $v_0$. We initialize the construction by setting $R_0 = \{v_0\}$. Since $\delta^+(G) = 7$, so $|R_1| = 7$. For $v_0$ to be a non-Seymour vertex, its second out-neighborhood must satisfy $|R_2| < |R_1|$, implying $|R_2| \le 6$.
		
		Let $G[R_1]$ be the subgraph induced by $R_1$. Suppose there exists a vertex $v_{1,1} \in R_1$ such that its out-degree within the induced subgraph is $d^+_{G[R_1]}(v_{1,1}) = 0$.  Then $v_{1,1}$ must possess at least 7 out-neighbors in $R_2$. This forces $|R_2| \ge 7$, which implies that the second out-degree of $v_0$ satisfies $|N^{++}(v_0)| \ge |R_1|$, making $v_0$ a Seymour vertex. Thus, every vertex $u \in R_1$ must satisfy $d^+_{G[R_1]}(u) \ge 1$.
		
		This degree constraint guarantees that $G[R_1]$ contains at least one directed cycle. This allows the nodes in $R_1$ to require fewer out-neighbors in $R_2$ to satisfy the minimum degree condition. Specifically, if every vertex in $R_1$ has $d^+_{G[R_1]}(u) \ge 2$, they collectively need at most 5 edges to distinct vertices in $R_2$, thereby maintaining $|R_2| \le 6$ and keeping $v_0$ non-Seymour. Because $G[R_1]$ contains 7 vertices, it can support at most $\binom{7}{2} = 21$ arcs, bounding the average induced out-degree by 3. This existence of internal cycles within $R_1$ is a necessary condition for $v_0$ to remain non-Seymour.
		
		We next examine the vertices in $R_2$, which similarly must satisfy $\delta^+(G) = 7$. We claim that every vertex in $R_2$ must have an internal out-degree $d^+_{G[R_2]}(u) \ge 2$. Suppose for contradiction that there exists a vertex $u_2 \in R_2$ with $d^+_{G[R_2]}(u_2) \le 1$. It follows that $u_2$ must have at least 6 out-neighbors in $R_3$.
		
		For vertices of sufficiently low degree, their neighborhood sizes double trivially under standard MCE minimality conditions \cite{KL}. 
		If a parent vertex $u_1 \in R_1$ of $u_2$ satisfies $|R_2| = |R_3|$, then the number of first forward out-neighbors of $u_1$ equals the number of its second forward out-neighbors. 
		This equality implies that $u_1$ is a Seymour vertex. Thus, we must have $d^+_{G[R_2]}(u) \ge 2$ for all $u \in R_2$.
		
		Furthermore, an induced subgraph $G[R_2]$ where every vertex has an out-degree strictly greater than 2 is impossible. It would require more arcs than is permitted by an oriented graph of its size. It follows that there exists at least one vertex in $R_2$ with an internal out-degree of exactly 2. This yields a neighborhood size of $|R_3| = 5$.
		
		To summarize, the cardinality of the rooted neighborhoods required to pack these non-Seymour vertices is given by:
		
		\begin{itemize}
			\item $|R_0| = |\{v_0\}| = 1$.
			\item $|R_1| = |\{v \in V(G) \mid \text{dist}(v_0, v) = 1\}| = 7$.
			\item $|R_2| = |\{v \in V(G) \mid \text{dist}(v_0, v) = 2\}| \le 6$.
			\item $|R_3| = |\{v \in V(G) \mid \text{dist}(v_0, v) = 3\}| \le 5$.
		\end{itemize}
		
		Finally, consider the vertices in $R_3$. To prevent their respective parent vertices from becoming Seymour vertices, the vertices in $R_3$ must maintain an internal induced out-degree of $d^+_{G[R_3]}(v) \ge 3$. This configuration requires at least $3 \times |R_3| = 15$ arcs within $G[R_3]$. However, a simple graph on 5 vertices can sustain at most $\binom{5}{2} = 10$ arcs. These nodes will be unable to meet their necessary out-degree requirements through local $R_3$ cycles. Thus either back arcs (to nodes with more second neighbors) are necessary, or an equal or larger $R_4$. Either of these conditions creates Seymour vertices. 
		
		By this contradiction, the vertices in $R_3$ are guaranteed to be Seymour vertices.
		\end{example}
		\noindent\hrulefill
	
	\vspace{2.8ex}
	
	As the number of second neighbors is iteratively reduced, nodes must compensate by connecting to more of their own siblings. This allows their parents to become non-Seymour nodes. The out-neighbors must form forward facing transitive triangles. These triangles must repeatedly interlock to form cycles to reduce the parent's second neighborhood. 
	
	\begin{remark}
	Note that the construction above outlines a skeleton required for such an MCE environment. To rigorously satisfy all conditions of an MCE, many more proofs about the necessity of arcs would need to be adopted. We omit those here to preserve the clarity of the preceding example.
	\end{remark}
	
	\begin{lemma}[Internal Degree Lower Bound]\label{idlb}. 
		Let $G$ be an MCE under a BFS-Lex coordinate system. 
		Suppose that each layer $R_i$ is packed with at least $i$ disjoint cycles. This forces every vertex $u_i \in R_i$ to satisfy:
		$$d^+_{\text{$R_i$}}(u_i) \ge i$$
	\end{lemma}	
	\begin{proof}
		Let $u_i \in R_i$ be the first node to connect to $i$ nodes locally in $R_i$ and Let $u_{i-1} \in R_{i-1}$ be the node of minimum degree in $R_{i-1}$. By assumption, we know that $u_{i-1}$'s local degree will be Seymour through the union of disjoint cycles. Since $u_i \in R_i$ sends at most $i - 1$ nodes to $R_i$, it will send at least $\delta - (i - 1)$ nodes to $R_{i+1}$. This implies that $u_{i-1}$ will connect to $\delta - (i - 1)$ nodes in $R_i$ and $\delta - (i - 1)$ nodes in $R_{i+1}$. 
		
		We know that $u_{i-1}$ has $R_{i-1}$ degree of at least $i-1$. This implies that $u_{i-1}$ sends at most $\delta - (i-1)$ forward arcs to $R_i$. We see now that there are at most $|\delta (i - 1)$ nodes in $R_i$ and at least $\delta - (i-1)$ nodes in $R_{i+1}$. This causes $u_{i-1}$'s forward Seymour property to become valid. Since both its local and forward Seymour properties are true, its overall Seymour property is true. 
	\end{proof}
	
	\begin{corollary} [Decreasing Neighborhood Property]\label{dsp}
		The packing of non-Seymour nodes into disjoint cycles must lead to strictly decreasing neighborhood sizes after the root layer. 
		Thus, $|G[\text{dist}(k+1)]| < |G[\text{dist}(k)]|$ for all $k > 0$. 
	\end{corollary}
	\begin{figure}
		\centering
		\begin{tikzpicture}[>=latex, line width=0.9pt]
			
			\node[fill=green!70!black,draw, circle, inner sep=2pt] (v0) at (0, -0.5) {\(v_0\)};
			
			\def\n{7} \def\radius{1.5} \def\centerx{3.5}
			\foreach \i in {1,...,\n} {
				\node[fill=cyan, draw, circle, inner sep=1.2pt] (v1\i) at ({\centerx + \radius*cos(360/\n*(\i-1))}, {\radius*sin(360/\n*(\i-1))}) {\small \(v_{1,\i}\)};
			}
			
			\def\m{6} \def\rradius{1.2}
			\foreach \i in {1,...,\m} {
				\node[fill=magenta!40,draw, circle, inner sep=1.2pt] (v2\i) at ({2*\centerx + \rradius*cos(360/\m*(\i-1))}, {\rradius*sin(360/\m*(\i-1))}) {\small \(v_{2,\i}\)};
			}
			
			\def\p{5} \def\rrradius{0.8}
			\foreach \i in {1,...,\p} {
				\node[fill=orange,draw, circle, inner sep=1.2pt] (v3\i) at ({3*\centerx + \rrradius*cos(360/\p*(\i-1))}, {\rrradius*sin(360/\p*(\i-1))}) {\small \(v_{2,\i}\)};
			}
			
			\foreach \i in {1,...,\n} {
				\draw[lightgray,->] (v0) -- (v1\i);
			}
			
			\foreach \i in {1,...,\n}{
				\foreach \j in {1,...,\m} {
					\draw[lightgray,->] (v1\i) -- (v2\j);
				}
			}
			
			\foreach \i in {1,...,\m}{
				\foreach \j in {1,...,\p} {
					\draw[lightgray,->] (v2\i) -- (v3\j);
				}
			}
			
			\foreach \i [evaluate=\i as \next using {int(mod(\i,\n)+1)}] in {1,...,\n} {
				\draw[->, bend right=15] (v1\i) to (v1\next);
			}
			
			\foreach \i [evaluate=\i as \next using {int(mod(\i,\m)+1)}, 
			evaluate=\i as \afternext using {int(mod(\i+1,\m)+1)}] in {1,...,\m} {
				\draw[->] (v2\i) to (v2\next);
				\draw[->] (v2\i) to (v2\afternext);
			}
			
			\foreach \i [evaluate=\i as \next using {int(mod(\i,\p)+1)}, 
			evaluate=\i as \afternext using {int(mod(\i+1,\p)+1)}, 
			evaluate=\i as \afterafternext using {int(mod(\i+1,\p)+1)}] in {1,...,\p} {
				\draw[->] (v3\i) to (v3\next);
				\draw[->] (v3\i) to (v3\afternext);
				\draw[->] (v3\i) to (v3\afterafternext);
			}
			
			\def\n{7} \def\radius{1.5} \def\centerx{3.5}
			\foreach \i in {1,...,\n} {
				\node[fill=cyan, draw, circle, inner sep=1.2pt] (v1\i) at ({\centerx + \radius*cos(360/\n*(\i-1))}, {\radius*sin(360/\n*(\i-1))}) {\small \(v_{1,\i}\)};
			}
			
			\def\m{6} \def\rradius{1.2}
			\foreach \i in {1,...,\m} {
				\node[fill=magenta!40,draw, circle, inner sep=1.2pt] (v2\i) at ({2*\centerx + \rradius*cos(360/\m*(\i-1))}, {\rradius*sin(360/\m*(\i-1))}) {\small \(v_{2,\i}\)};
			}
			
			\def\p{5} \def\rrradius{0.8}
			\foreach \i in {1,...,\p} {
				\node[fill=orange,draw, circle, inner sep=1.2pt] (v3\i) at ({3*\centerx + \rrradius*cos(360/\p*(\i-1))}, {\rrradius*sin(360/\p*(\i-1))}) {\small \(v_{3,\i}\)};
			}		
		\end{tikzpicture}
		\caption{\textit{BFS \& Lex Graph in an MCE. The minimum out-degree is 7. Nodes are arranged according to their BFS layers. Each node within these layers is assigned a distinguishing distance \& lexicographic ID.} } \label{earlyssnc}
	\end{figure}

	\section{Graph Level Order} \label{glover}
	\subsection{Introduction}
	
	Graph theory is the study of relationships. The Graph Level Order introduces set-theoretic ancestry to those connections. After identifying the limitations of existing terminology, we formally define the Graph Level Order architecture. This includes new set-theoretic types of neighbors, partitions of those neighbors, and theorems mapping these terms to existing terminology. This systematic reorganization provides the clarity necessary to finally resolve the SSNC.
	
	Standard neighborhood definitions suffer from an accounting failure: they cannot be partitioned without destroying ancestral data. When a node's influence is sent to its first neighbors, the connection to its second neighbors is reduced to a simple count. This strips distance-two neighbors of any knowledge of their specific parentage. This loss of context creates a disconnect and makes it impossible to track the precise distribution of out-degree influence.
	
	This counting problem is demonstrated in the following case:
	
	Let $u_i$ be a parent node, and let $v_{i+1}, w_{i+1} \in N^+(u_i)$ be distinct children of $u_i$. Suppose there exists a vertex $y_{i+2}$ such that $$ v_{i+1} \to y_{i+2} \in G \quad \text{and} \quad w_{i+1} \to y_{i+2} \in G, \text{ but } u_i \not\to y_{i+2} \in G. $$ Then the quadruple $(u_i, v_{i+1}, w_{i+1}, y_{i+2})$ forms a \textbf{Seymour diamond}. 
	
	\begin{figure}[ht] 
		\centering
		\begin{tikzpicture}[node distance=2cm] 
			\node (u_i) at (0, 0) {$u_i$};
			\node (v_{i+1}) at (2, 2) {$v_{i+1}$};
			\node (w_{i+1}) at (2, -2) {$w_{i+1}$};
			\node (y_{i+2}) at (4, 0) {$y_{i+2}$};
			\draw[->][thick] (u_i) -- (v_{i+1}); 
			\draw [->](u_i) -- (w_{i+1});
			\draw[->][thick] (v_{i+1}) -- (y_{i+2});
			\draw [->](w_{i+1}) -- (y_{i+2});
			
		\end{tikzpicture}  
		\caption{\textit{\textbf{A Seymour diamond}. A parent node $u_{i}$ with children $v_{i+1}$ and $w_{i+1}$, and a shared grandchild $y_{i+2}$ resulting in double counting in neighborhood analysis. The sum of the children's out-degrees is $|N^{+}(v_{i+1})|+|N^{+}(w_{i+1})|$. The actual contribution to the second out-neighborhood $|N^{++}(u_{i})|$ is only one, as both arcs terminate at the same node, $y_{i+2}$.} }
		\label{smdmpic}
	\end{figure}
	
	Seymour diamonds double-count second neighbors. 
	When two children at distance one, $v_{i+1}$ and $w_{i+1}$, both point to the same node $y_{i+2}$, a summation of their out-degrees fails. The operation $|N^+(v_{i+1})| + |N^+(w_{i+1})|$ incorrectly counts the same node twice. 
	
	The \textbf{rooted neighborhood} at distance $i \ge 0$ from $v_0$ is the subgraph induced by the vertices at directed distance $i$ from $v_0$. 
	
	A triple of vertices $x,y,u$ in an oriented graph $G$ forms a \textbf{transitive triangle} if $$ x \to y,\qquad x \to u,\qquad y \to u, $$ i.e., the three vertices induce an acyclic orientation of a 3-cycle.
	
	Let $u_i \in R_i$ be a parent node in the rooted layering, and let $v_{i+1} \in R_{i+1}$ be one of its children. The \textbf{interior neighbors} of $v_{i+1}$ with respect to $u_i$ are the vertices $$ \text{int}(u_i, v_{i+1}) = N^+(u_i) \cap N^+(v_{i+1}) \cap R_{i+1}. $$ The \textbf{interior degree} of $v_{i+1}$ with respect to $u_i$ is $$ |\text{int}(u_i, v_{i+1})|. \label{int_nbr} $$ 
	
	
	Let $u_i \in R_i$ be a parent node, and let $v_{i+1}, w_{i+1} \in R_{i+1}$ be children of $u_i$. Then $v_{i+1}$ and $w_{i+1}$ are called \textbf{siblings}. \label{siblings}
	
	Let $u_i \in R_i$ be a parent of $v_{i+1} \in R_{i+1}$. The \textbf{exterior neighbors} of $v_{i+1}$ with respect to $u_i$ are the vertices $$ \text{ext}(u_i,v_{i+1}) = N^{++}(u) \cap N^+(v_{i+1}) \cap R_{i+2}. $$
	The \textbf{exterior degree} of $v_{i+1}$ with respect to $u_i$ is $$ |\text{ext}(u_i, v_{i+1})|. $$\label{ext_nbr}
	
	Referring back to Figure \ref{earlyssnc}, $v_0$ is the MDN. The nodes at distance one, the first cycle, are the interior neighbors of $v_0$ with respect to another node in the cycle. They are all siblings because they are all in the same rooted neighborhood. The nodes at distance two from $v_0$, the second cycles, are the exterior neighbors of $v_0$ with respect to any particular interior neighbor. 
	
	\begin{remark} For the MDN $v_0$, all its out-neighbors are exterior neighbors because $R_0 = \{v_0\}$ and there are no earlier neighborhoods $R_{-1}$. Thus, $$ \text{int}(v_0, v) = \emptyset \hspace{1cm} \forall v \in\space R_1 $$ \end{remark}
	
	Let $v_0$ be the MDN, and let $u_i \in R_i$ be a vertex in $G$ with an earlier neighbor $v_j \in R_j$. The set of \textbf{back arcs} from $u_i$ to $v_j$ is $\text{back}(u_i, v_{i+1})$. The set of all back arcs of $u_i$ is denoted $\bigcup_{j < i} R_j \text{back}(u_i)$. \label{back_def}
	
	\begin{lemma}[Degree Partition]\label{nodepartition}
		Let $u_i \in R_i$ and let $v_{i+1} \in R_{i+1}$ with $u_i \to v_{i+1}$. Then the out-neighborhood of $ v_{i+1}$ admits the disjoint partition
		
		$$N^+(v_{i+1}) = \text{back}(u_i, v_{i+1}) \;\dot{\cup}\; \text{int}( u_i, v_{i+1}) \;\dot{\cup}\; \text{ext}( u_i, v_{i+1}).$$ \end{lemma}
	\begin{proof}
		Every out-neighbor of $ v_{i+1}$ lies in exactly one rooted neighborhood. 
		Arcs may point only to vertices in layers $R_j$ for $j < i$ , $R_{i}$, $R_{i+1}$, or $R_{i+2}$.
		Thus any $x \in N^+( v_{i+1})$ lies in one of these layers. 
		\begin{description}
			\item If $x \in R_j$ for $j < i+1$, then $x \in \text{back}(u_i, v_{i+1})$. 
			\item If $x \in R_{i+2}$, then $x \in \text{ext}( u_i, v_{i+1})$.   
			\item If $x \in R_{i+1}$, then $x \in \text{int}( u_i, v_{i+1})$. These sets are disjoint because distance is disjoint and cover all of $N^+( v_{i+1})$.
		\end{description}
	\end{proof}
	
	Lemma \ref{nodepartition} (Degree Partition) establishes a stronger linkage than first and second neighbors. By decomposing the out-degree into three mutually exclusive parts, the SSNC is transformed from a broad search into a targeted analysis. We can now analyze how each component—interior, exterior, and back—must simultaneously be non-Seymour for an MCE to exist. This decomposition provides the basis for our upcoming proofs.
	
	\begin{definition}\label{glover1} 
		Let $G = (V,E)$ be an oriented graph, and let $v_0 \in V$ be the vertex of minimum out-degree. A \textbf{Graph Level Order} on $G$ consists of the following structure: 
		\begin{enumerate} 
			\item Well Ordered Rooted Neighborhoods: Partition the vertices of $V$ into rooted neighborhoods $$R_i=\{v \in V : \text{dist}(v_0,v)=i\}, \text{ for } i=0,1,\dots,n.$$ 
			\item Well Ordered Vertices: For vertices $u_{i,k_1}, u_{i,k_2} \in R_i$ and $v_{j,m_1}, v_{j,m_2} \in R_j$ with $i<j$, we have $u_{i,k_1}, u_{i,k_2}<v_{j,m_1}, v_{j,m_2}$ and $u_{i,k_1} < u_{i,k_2}$. 
			
			\item Ancestry-Awareness: For each parent-child pair $u_i \to v_{i+1}$ with $u_i \in R_i$ and $v_{i+1} \in R_{i+1}$, the out-neighbors of $v_{i+1}$ are classified into the three disjoint sets according to Lemma \ref{nodepartition}: $\text{back}(u_i, v_{i+1})$, $\text{int}(u_i,v_{i+1})$, and $\text{ext}(u_i,v_{i+1})$. 
		\end{enumerate}
	\end{definition}
	
	\subsection{Using Graph Level Order}
	
	With the Graph Level Order vocabulary established, these definitions now must connect to the mathematical foundation. Interior neighbors relate to the first neighbors of a parent. Exterior neighbors relate to the second neighbors. In an MCE environment, Theorem \ref{lbl} (Load Balance) shows the only purpose of these interior neighbors is to be forward-facing transitive triangles that reduce the parent’s second neighborhood. 
	
	An edge $e$ is \textbf{unnecessary} to the MCE $G$ if replacing $e$ with an forward arc $e'$ in the MCE results in a graph that preserves the non-Seymour status of all vertices ($|N^{++}_{G'}(v)| < |N^{+}_{G'}(v)|$ for every $v \in V$).
	\begin{theorem}[Load-Balance]\label{lbl}
		Let $G \in \mathcal{O}$ be an MCE. Then the first out-neighborhood of every parent $u$ is a minimal set cover if the interior neighborhoods of its children: 
		$$N^+(u) = \bigcup_{v \in N^+(u)} \text{int}(u, v)$$
		This implies that $u$ is non-Seymour and we can reduce the second neighborhood size by 1. Further, extra interior arcs are unnecessary. 		
	\end{theorem}
	\begin{proof} 
		Let $x \in N^+(u)$. By assumption, there exists $v \in N^+(u)$ such that $v \to x \in G$. Then $x \in N^+(v)$, and hence $x \in \text{int}(u,v)$. Since $x$ was arbitrary, we have $$ N^+(u) \subseteq \bigcup_{v \in N^+(u)} \text{int}(u,v). $$ Conversely, for any $x \in \text{int}(u,v)$ with some $v \in N^+(u)$, we have $x \in N^+(u)$ by definition. Thus, $$ \bigcup_{v \in N^+(u)} \text{int}(u,v) \subseteq N^+(u). $$ Combining both inclusions, we conclude $$ N^+(u) = \bigcup_{v \in N^+(u)} \text{int}(u,v). $$ 
		
		Suppose the cover is not minimal. Then at least one node $v \in N^+(u)$ has out-degree at least 2. However, since only the out-degree of 1 by each node is needed to cover each node in $N^+(u)$, that extra arc by $v$ is unnecessary. 
		
		MCEs are supposed to be minimum density. We need to show that swapping to $e = v_i \to w_i$ a forward arc keeps the ancestor nodes, parent node, child and downstream nodes non-Seymour. The swap of $v_i \to w_i$ to $v_i \to w_{i+1}$ is safe for $u$ by the definition of an unnecessary arc. It does not hurt the parent node. It is also safe for the node $w_{i+1}$ as it is an incoming arc into $w_{i+1}$ and SSNC only cares about outgoing arcs. Also, in this dynamic state, the node $w_{i+1}$ has not been processed yet. Nodes in $N^+(v_k)$ have not been processed yet either because they are at a further distance $(k+1)$. Similarly, ancestor nodes at earlier distances $\text{dist}(v_0, v_k) < k$ are already proven non-Seymour and unaffected. 
		
		Finally, this swap $v_i \to w_i$ to $v_i \to w_{i+1}$ lowers the density of the graph because the first $i$ rooted neighborhoods have the same number of nodes, but will have one fewer arc. This shows that extra interior arcs cannot be a part of a minimal density MCE. 
		
		This forces every node in $N^+(u)$ to have out-degree of 1.  So number of arcs that each child node $v \in N^+(u)$ sends to $N^{++}(u)$ can be reduced by exactly 1. Meaning that the overall size of $N^{++}(u)$ can be reduced by exactly 1. 
		
		The result of this is that $|N^+(u)| - |N^{++}(u)| = 1$, while still keeping $u$ a non-Seymour node. 
		
	\end{proof}
	
	Theorem \ref{lbl} Load Balance restricts interior nodes to the cycle packing role. Every child $v \in N^+(u_i) \cap R_{i+1}$ directs an arc toward a sibling in $w \in R_{i+1}$. This distributes $u_i$'s load so no node's second neighborhood has to expand. The theorem reveals a dependency: parents rely on children to interlock. This preserves non-Seymour nodes.
	
	\begin{figure}
		\centering
		\begin{tikzpicture}[scale=0.9]
			\node[fill=green!70!black, draw, circle, inner sep=2pt] (v0) at (0, -0.5) {\(u\)};
			
			\def\n{7} 
			\def\radius{1.3}
			\def\centerx{3.5}
			
			\foreach \i in {1,...,\n} {
				\node[fill=cyan, draw, circle, inner sep=1.2pt] (v1\i) at ({\centerx + \radius*cos(360/\n*(\i-1))}, {\radius*sin(360/\n*(\i-1))}) {\small \(v_{i,\i}\)};
			}
			
			\foreach \start/\target in {4/5, 5/6, 6/7, 7/1, 1/2, 2/3, 3/4}
			{
				\draw [lightgray, ->, thick] (v0) -- (v1\start);
				\draw [lightgray, ->, thick] (v0) -- (v1\target);
				\draw [cyan, ->, thick] (v1\start) -- (v1\target);
			}
			
			Using red and dashed to signify "This is banned"
			
			\node[anchor=north] at (\centerx, -2) {
			};
		\end{tikzpicture}
		\caption{\textit{\textbf{Set cover implies a cycle}. Nodes have out-degree = 1 in $N^+(u)$. This implies that every node can have one less out-degree in $N^{++}(u)$. Thus, the size of $N^{++}(u)$ can be reduced by 1. }}
	\end{figure}
	
	\begin{lemma}[Next Link] \label{link}
		Let $u_i$ be a node in $G$ that is non-Seymour.	
		If a back arc $u_i \to v$ exists, then either $u_i$ is Seymour or there exists another back arc $u' \to v'$ such that $u' \in N^+(u_i)$ and $v' \in N^+(v)$ that leads to a Seymour vertex.
	\end{lemma}
	\begin{proof}
		Let $u_i \to v_k$ be a back arc, with $u_i \in R_i$ and $v_k \in R_k$ such that $k < i$.
		Since $u_i$ is non-Seymour, we have $|N^{++}(u_i)| < |N^+(u_i)|$.
		The back arc $u_i \to v_k$ creates paths of length two: $u_i \to v_k \to w$ for every $w \in N^+(v_k)$.
		Thus, $N^+(v_k) \subseteq N^{++}(u_i)$.
		Partition $N^+(v_k)$ into:
		\[
		I_1 = N^+(u_i) \cap N^+(v_k), \quad
		I_2 = N^+(v_k) \setminus I_1.
		\]
		
		Because $v_k \in R_k$ where $k < i$, the elements of $N^+(v_k)$ reside in neighborhoods preceding the reach of $N^{++}_{old}(u_i)$ relative to the distance from $v_0$. Thus, $I_2$ and $N^{++}_{old}(u_i)$ are completely disjoint.
		
		Case 1: Seymour Property.
		If $|I_1| < |N^+(u_i)| - |N^{++}_{old}(u_i)|$, then the second neighborhood of $u_i$ grows too large:
		$$|N^{++}(u_i)| = |N^{++}_{old}(u_i)| + |I_2| \ge |N^+(u_i)|,$$
		satisfying the Seymour property. 
		
		Case 2: Inductive Chain.
		Assume the back arc $u_i \to v_k$ does not immediately create a Seymour vertex.
		Then $|I_1| \ge |N^+(u_i)| - |N^{++}_{old}(u_i)|$. Suppose, for contradiction, that no $u' \in N^+(u_i)$ has a back arc to any $w \in N^+(v_k)$.	
		Then all $w \in N^+(v_k)$ are new second neighbors of $u_i$, disjoint from $N^{++}_{old}(u_i)$, implying
		$$|N^{++}(u_i)| = |N^{++}_{old}(u_i)| + |N^+(v_k)| \ge |N^+(u_i)|,$$
		which is Seymour.
		
		Thus, there must exist at least one $u' \in N^+(u_i)$ such that a back arc $u' \to w$ exists for some $w \in N^+(v_k)$.
		
		\begin{itemize}
			\item $u' \in N^+(u_i)$, so it is a child of $u_i$.
			\item $w \in N^+(v_k)$, so it is a neighbor of $v_k$.
		\end{itemize}
		
		This establishes the existence of a second back arc extending from an out-neighbor of $u_i$ to a neighbor of $v_k$, completing the next link in the chain.
		
		Hence, any back arc either directly leads to a Seymour vertex or leads to a chain of back arcs to a Seymour vertex.
	\end{proof}
	
	\begin{theorem}[Reduction]\label{red}
		Let $G \in \mathcal{O}$ be an MCE. For any $u \in R_i$, the second out-neighborhood is the maximal union of the exterior neighborhoods of its children: $$N^{++}(u) = \left( \bigcup_{v \in N^+(u)} \text{ext}(u,v) \right) \cup \left( \bigcup_{v \in N^+(u)} \text{back}(v) \right).$$
		This implies that back arcs are unnecessary.
	\end{theorem}
	\begin{proof} To establish the equality $N^{++}(u) = \left( \bigcup \text{ext}(u,v) \right) \cup \left( \bigcup \text{back}(u, v) \right)$, we show containment in both directions.
		
		($\subseteq$): Let $x \in N^{++}(u)$. By definition, there exists a path $u \to v \to x$. Since $u \in R_i$, its neighbor $v$ must be in $R_{i+1}$. By the properties of rooted neighborhoods, any neighbor $x$ of $v$ must lie in $R_j$ where $j \le (i+1) + 1 = i+2$. Because $x$ is a second neighbor of $u$, it cannot lie in $N^+(u) = R_{i+1}$, nor can it be $u$ itself. Thus, $x$ must either lie in $R_{i+2}$ or in some neighborhood $R_j$ with $j < i$. If $x \in R_{i+2}$, then $x \in \text{ext}(u,v)$. If $x \in R_j$ for $j < i$, then $x \in \text{back}(u, v)$. In either case, $x$ is contained in the union of the children's exterior and back neighborhoods.
		
		($\supseteq$): Let $x$ be a back node of some child $v \in N^+(u)$. There exists a path $u \to v \to x$. Since $x$ is in a neighborhood prior to $R_i$, it cannot be in $N^+(u)$ or be $u$ itself. Thus, $x \in N^{++}(u)$. Let $x$ be an exterior neighbor of some child $v \in N^+(u)$. There exists a path $u \to v \to x$ where $x \in R_{i+2}$. Since $R_{i+2}$ is disjoint from $R_{i+1}$ and $R_i$, $x$ is neither a neighbor nor the parent, meaning $x \in N^{++}(u)$. Combining these, we conclude that the second neighborhood is precisely the union of these two sets.
		
		Lemma \ref{link} (Next Link) says that the presence of a single back arc induces Seymour vertex or a path to a Seymour vertex. This makes it impossible to support back arcs in an MCE. As a result, back arcs in an MCE are unnecessary. 
		
		To show the elements of $N^+(u)$ provide a maximal set cover. Theorem 1 (Load Balance) implies every node in $u$'s out-neighborhood needs to form cycles. This restriction assures that every node in $N^+(u)$ sends the same number of arcs to $N^{++}(u)$, its full exterior capacity. No child can contribute more without violating non-Seymour properties. No additional sets exist beyond those indexed by $N^+(u)$. This coverage is maximal with respect to the available exterior capacity of children.
	\end{proof}
	
	Theorems \ref{lbl} (Load Balance), \ref{link} (Next Link) and \ref{red} Reduction transform the conjecture from a counting problem of degree sums into a set packing problem. A minimum counterexample environment tries to pack non-Seymour nodes into a minimally dense environment. An MCE tries to keep this process going as long as possible. Without back arcs to loop around, this is a strictly decreasing set of natural numbers that is bounded above. This is, by definition, a bounded set. 
	
	Note that the elimination of these back arcs does not imply that the graph becomes a directed graph containing a global sink. By Theorem \ref{lbl} (Load Balance), every node in the first neighborhood of the parent still belongs to a cycle. 
	
	\begin{theorem}[Graph Level Order Representation of Oriented Graphs]\label{gloverrep}
		Every oriented graph $G = (V,E)$ can be represented by a Graph Level Order. \end{theorem} 
	\begin{proof}
		We present an algorithm for this procedure. 
		\begin{enumerate}
			\item Select a minimum out-degree root
			\item Define rooted neighborhoods. For each $i \ge 1$, define $$ R_i = \{ u_i \in V : \text{dist}(v_0, u_i) = i \}, $$ 
			\item Total order within neighborhoods. Within each $R_i$, impose a consistent total order on vertices (e.g., lexicographic or another ordering). 
			\item Parent–child relations. For $u_i \in R_i$ and $v_{i+1} \in R_{i+1}$, if $u_i \to v_{i+1} \in E$, declare $u_i$ a parent and $v_{i+1}$ a child. This defines interior and exterior neighbors $u_i \to v_{i+1}$: $$ \text{int}(u_i,v_{i+1}) = N^+(u_i) \cap N^+(v_{i+1}) \cap R_{i+1}, \quad \text{ext}(u_i,v_{i+1}) = N^{++}(u_i) \cap N^+(v_{i+1}) \cap R_{i+2}. $$
		\end{enumerate}
		By construction, every out-neighbor of $v$ in $R_{i+1} \cup R_{i+2}$ is either interior or exterior relative to its parent $u$. 
		
		\medskip Thus, $G$ is fully described by the rooted neighborhoods $R_0, \dots, R_k$, the parent–child relations, and the interior/exterior neighbor sets. This satisfies all the properties required for a Graph Level Order representation.     
	\end{proof}
	
	Because of this universal representation of graphs, when we speak of oriented graphs we are always talking of a Graph Level Order environment. 
	
	These next two results follow directly from the definitions of regularity and unnecessary arcs. 
	
	\begin{corollary}[Regular Graph]\label{reggraph}
		Let $G \in \mathcal{O}$ be an MCE to the SSNC. All nodes in $G$ have out-degree $\delta$. 
	\end{corollary}
	
	\begin{corollary}[Regular Interior Degrees]\label{regint}
		Let $G \in \mathcal{O}$ be an MCE to the SSNC, and let $R_k$ be a rooted neighborhood at distance $k$. Then every node $v_k \in R_k$ will have an interior out-degree of $k$ in $R_k$. \end{corollary}
	
	The following corollaries help to establish the neighborhood density. They follow from the Load Balance Theorem \ref{lbl} applied first on consecutive neighborhoods, then inductively on the entire graph. 
	
	\begin{corollary}[Decreasing Exteriors] \label{del}
		Let $G \in \mathcal{O}$ be an MCE to the SSNC and 	
		let $u_i \in R_i$ and let $v_{i+1} \in N^+(u_i) \cap R_{i+1}$.
		Then for every $z \in \mathrm{ext}(u_i,v_{i+1})$,$$|\mathrm{ext}(v_{i+1},z)| < |\mathrm{ext}(u_i,v_{i+1})|.$$
	\end{corollary}
	
	\begin{corollary}[Impact on Rooted Neighborhood Size] \label{nbhsize}
		Let $G \in \mathcal{O}$ be be an MCE, with rooted neighborhoods $R_0, R_1, \dots, R_k$. Then for all $i \ge 0$,
		$$|\text{ext}(u_i, v_{i+1})| \le |N^+(v_0)| - i\text{ where }u_i \in R_i, v_{i+1} \in R_{i+1}$$
	\end{corollary}
	
	\begin{corollary}[Formula for Rooted Neighborhood Size] \label{nbhsizefmla}
		Let $G \in \mathcal{O}$ be an MCE. The rooted neighborhoods $R_i$ satisfy the bounds:
		$$|R_0| = 1, \quad |R_1| = \delta, \quad |R_i| \le \delta - (i-1) \text{ for } i \ge 2.$$
	\end{corollary}
	
	This last formula cements the structure of our neighborhood density. For a given rooted neighborhood at distance $k$ in an MCE, it gives precise size calculations. 
	
	\begin{figure}
		\centering
		\begin{tikzpicture}[x=0.8cm, y=0.4cm, scale=0.8] 
			\foreach \h [count=\i] in {10, 9, 8, 7, 6, 5, 4, 3, 2, 1} {
				\draw[fill=blue!15, draw=blue!50] (\i, 0) rectangle (\i+0.6, \h);
				
				\ifnum\i=1 \node[below] at (\i+0.3, 0) {$R_1$}; \fi
				\ifnum\i=4 \node[below] at (\i+0.3, 0) {$R_j$}; \fi
				\ifnum\i=9 \node[below] at (\i+0.3, 0) {$u_i$}; \fi
			}
		\end{tikzpicture}
		\caption{\textit{\textbf{Iterated decrease in second neighborhood}. In a rooted environment, this means that there is a maximum value that every decrease comes from: $|R_k| = \delta - (k-1)$}}
	\end{figure}
	
	\begin{figure}
		\centering
		\begin{tikzpicture}[
			every edge/.style = {draw,->},
			Vertex/.style = {circle, draw, minimum size=1cm, inner sep=0pt},
			]
			\node[Vertex] (u_k) at (0, 0) {\(u_k\)};
			\node[Vertex] (v_{k+1}) at (2,0) {\(v_{k+1}\)};
			\node[Vertex] (w_{k+1}) at (2,2) {\(w_{k+1}\)};
			\node[Vertex] (z_{k+2}) at (4,0) {\(z_{k+2}\)};
			
			\draw[->] (u_k) -- (v_{k+1});
			\draw[->] (v_{k+1}) -- (z_{k+2});
			\draw[->] (v_{k+1}) -- (w_{k+1});
			\draw[->, dashed] (z_{k+2}) -- (w_{k+1});
			
		\end{tikzpicture}
		\caption{\textit{\textbf{False Flag problem of back arcs}.  Assuming that $z_{k+2} \to w_{k+1}$ is the first back arc, $w_{k+1}$ must map to a new exterior neighbor. This will make $u_k$ a Seymour vertex.} }\label{backexwk}
	\end{figure}
	
	The rooted structure of the Graph Level Order is an important characteristic. Consider Example \ref{backexwk}. At first glance, the back arc $z_{k+2} \to w_{k+1}$ may seem to violate the principles of the Graph Level Order. However, this does not take into account their respective distances from the root node. Assuming this is the first back arc, what we can show is that this back arc cannot exist. 
	
	Let $u_k$ be at distance $k$. Then $v_{k+1}$ and $w_{k+1}$ are at distance $k+1$ and $z_{k+2}$ is at distance $k+2$. The nodes $v_{k+1}$ interior degree must be $k+1$, implying they have $\delta-k-2$ exterior neighbors. One of $v_{k+1}$'s exterior neighbors is $z_{k+2}$, but $w_{k+1}$ cannot relate to $z_{k+2}$ by definition of an oriented graph. This node $w_{k+1}$ must still have $\delta-k-2$ exterior neighbors. So $w_{k+1}$ must relate to a new node $y_{k+2}$. This means that $u_k$'s second neighborhood did not decrease by one as the Load Balance Theorem \ref{lbl} says it must by the cycles in its first neighborhood. Thus, $u_k$ will be is a Seymour vertex. 
	
	\subsection{Power of the Data Structure}\label{power}
	This would be a futile attempt at proving the SSNC if we could not manage the complications mentioned in previous papers. First are transitive triangles, where Brantner et al. \cite{BBKS} showed that oriented graphs without transitive triangles contain a Seymour vertex, the MDN.
	
	\begin{theorem}[Load-Balancing Triangles]\label{lbt}The only possible transitive triangles in an MCE are load-balancing triangles.\end{theorem}
	\begin{proof}
		We classify all transitive triangles by examining the first two arcs and determine whether a valid third arc can close the triangle.
		
		A triangle has three arcs, so there are $2^3 = 8$ possible combinations of arc types. We enumerate these cases in Table \ref{transtricases}:
		\begin{table}[h!] 
			\centering 
			\begin{tabular}{|c|c|c|c|} \hline Case
				& Labels & Triangle Possible? & Reason
				\\ \hline 
				1 & 3I & Yes & (all Intra-layer)\\ 
				\hline 
				2, 4 & 1B2I / 1F2I  & No & Symmetric Arcs \\ 
				\hline 
				3 & 2B1I & Yes & $x,y \in R_i$, $z \in R_{i+1},x \to y,z \to x,z \to y$\\ 
				\hline 
				5 & 1F1B1I & Yes &$x, y \in R_i, z\in R_{i+1}, x  \to y,x \to z,z \to y$. \\ 
				\hline 
				6 & 1F2B & Yes & $x \in R_i ,y \in R_{i+1},z \in R_{i+2},x \to y$,$z \to x$,$z \to y$. \\ 
				\hline 
				7 & 2F1I & Yes &$x\in R_i,y,z \in R_{i+1},x \to y, x \to z, y \to z$.  \\
				\hline 
				8 & 2F1B & No & \makecell{cannot form a closed transitive triangle.} \\  
				\hline 
			\end{tabular} \caption{Eight possible configurations of transitive triangles in a Graph Level Order.} 
			\label{transtricases} 
		\end{table}
		
		From these 8 possible cases, three are excluded immediately because of impossibility (Cases 2, 4, 8). Theorem \ref{red} (Reduction) says that back arcs are unnecessary in an MCE (Cases 3, 5, and 6). 
		
		The remaining triangles are the following: Case 1: $x \to y \to z \to x$, all in $R_i$, and the Load-Balancing Triangle (Case 7). A triangle from Case 1 is not a load-balancing triangle, so it contains an extra interior arc. By Theorem \ref{lbl} (Load Balance) $G$ cannot be an MCE.
	\end{proof}
	
	\begin{theorem}[Resolution of Seymour Diamonds]\label{smdiam}
		Let $G \in \mathcal{O}$ be a Graph Level Order containing a Seymour diamond. Then the ambiguity is resolved via exterior neighbors. \end{theorem}
	
	\begin{proof}
		Assume $G$ has a Seymour diamond $u_i, v_{i+1}, w_{i+1},y_{i+2}$. 
		By theorem \ref{red} (Reduction), 
		$$N^{++}(u_i) = \bigcup_{v_{i+1} \in {N^+(u_i)}} \text{ext}(u_i, v_{i+1}).$$
		
		The exterior neighbors are sets, so any node reached by multiple length-2 paths is counted only once in $N^{++}(u_i)$. 
		This eliminates double-counting for the node $u_i$. 
	\end{proof}
	
	The two main hindrances of this conjecture have been shown that they are not roadblocks. The data structure used throughout this paper has been defined. This data structure provides an understanding of interior and exterior arcs, and guides the search for a counterexample. It also provides understanding of non-Seymour nodes and guides the nodes in this proof process.
	
	\section{Global Impact of Packing} \label{packingNS}
	\subsection{Introduction}
	
	By definition, an MCE is always on the edge of collapse. Maximizing the number of non-Seymour nodes formalizes this fragility. Although it is invoked as a contradictory hypothesis, the argument is inherently constructive. The graph was set up based on distance. Set-theoretic, ancestral neighborhoods were constructed to resolve past obstructions. This is explicitly explaining why these things happened. 
	
	Let $G \in \mathcal{O}$ be an oriented graph represented by a Graph Level Order. Define $v_k \sim w_k$ for two vertices $v_k, w_k \in R_k$ if there exists a rooted automorphism of $G$ fixing $v_0$ and preserving rooted neighborhoods $R_i$ such that $v_k \mapsto w_k$
	
	\begin{theorem} [Rooted Neighborhood Quotient]\label{nbhquot} 
		Let $G$ be an MCE. 
		Every vertex within a rooted neighborhood $R_k$ is indistinguishable under the rooted automorphism $\sim$. 
		The resulting quotient graph $G/\sim$ is a well-defined hyper-node graph that preserves the out-degree and non-Seymour of $G$.    
	\end{theorem}
	\begin{proof}
		Let $v_k, w_k \in R_k$. By Theorem \ref{bfslex} these nodes have identical parents because they are at the same distance. Lemma \ref{nodepartition} (Degree Partition) says that $v_k$'s and $w_k$'s degree can be partitioned into interior and exterior degree. Corollary \ref{regint} (Regular Interior Degrees) says that both nodes have identical interior degrees of $k$. Similarly, since back arcs are excluded by Theorem \ref{red} (Reduction), they possess identical exterior neighbors. Furthermore, by Theorem \ref{lbl} (Load Balance), all interior arcs of $v_k$ and $w_k$ must induce forward-facing transitive triangles. As these are the unique load-balancing triangles permitted by Theorem \ref{lbt} (Load-Balancing Triangles), the rooted induced subgraphs $G[R(v_k)]$ and $G[R(w_k)]$ are isomorphic under a neighborhood-preserving automorphism fixing the root. Since this isomorphism holds for all pairs in $R_k$, the set $R_k$ constitutes an equivalence class under automorphism. Collapsing these classes yields a well-defined quotient graph $G / \sim$.
	\end{proof}
	
	The rooted neighborhood quotient structure can be seen in Figure \ref{earlyssnc}. As shown in Section \ref{glover}, every node in the graph has an out-degree of 7. By neighborhood, all nodes at distance one have interior out-degree of one. All nodes at distance two have an interior out-degree of two. Nodes having the same distance from the root, interior and exterior exterior neighbor set, making them automorphic. 
	
	Further, we see that each of these nodes in the first two neighborhoods and the root are also non-Seymour. This causes the neighborhoods to decrease in size from seven nodes to six to five as Corollary (Impact on Rooted Neighborhood Size) \ref{nbhsize} predicted. 
	
	The root serves as a driving force, influencing nodes and setting conditions for further analysis. In contrast, non-Seymour nodes impose a ceiling. This concept is like a bottleneck. These are  the forces that fold the graph into its linear hyper-node shape. 
	
	\subsection{Law of Conservation of Neighbors}
	
	Let $u_k \in R_k, v_{k+1}, w_{k+1} \in R_{k+1}$. Define $$T(u_k):=\{(v_{k+1},w_{k+1}) \in E | v_{k+1},w_{k+1} \in N^+(u_k),v_{k+1} \to w_{k+1}\}$$ denote the set of interior arcs among the children of $u_k$ in $R_{k+1}$. 
	
	\begin{theorem}[Law of Conservation of Neighbors] \label{connei}
		Let $G \in \mathcal{O}$ be an MCE to the SSNC with $u_k \in R_k$. 
		
		In a minimal counterexample to the SSNC, there is a direct trade-off between exterior and interior neighbors. 
		\begin{enumerate}
			\item For any parent node $u_k$ at level $k$: Each unit decrease in the exterior neighborhood (vertices in $R_{k+2}$) forces an increase in the interior degree (arcs within $R_{k+1}$).
			\item Because these interior arcs connect siblings, they create forward-facing transitive triangles rooted at $u_k$.
			\item The total number of these triangles is governed by:$$|T(u_k)| = |R_{k+1}| \cdot (k+1)$$
		\end{enumerate}
	\end{theorem}
	
	\begin{proof}
		Let $G \in \mathcal{O}$ be an MCE. The children of $u_k$, which belong to $R_{k+1}$ are a union of $R_{k+1}$ and each of their interior degrees is $k+1$. By theorem Load Balance \ref{lbl}, each child is in a cycle in $R_{k+1}$ that causes $u_k$'s second neighborhood to contract by 1. Each of these children forms a transitive triangle with $u_k$, forming $k+1$ transitive triangles. 
		Each cycle in $R_k$ is composed of forward facing transitive triangles between the parent node from $R_{k-1}$ and two children of $R_k$. This means that the number of transitive triangles is the degree of each node multiplied by the number of times nodes must appear in a cycle, $$|T(u_k)| = |R_{k+1}| \cdot (k+1).$$
		This can also be computed by individually counting cycles $C_1, C_2, ..., C_k, C_{k+1}$ in $R_{k+1}$. Each cycle has degree $|R_{k+1}|$ and we arrive at this same number when we sum it $k+1$ times. 
		
	\end{proof}
	
	The Law of Conservation of Neighbors \ref{connei} shows a sequence of rooted neighborhoods that are each defined by their nodes and the implied cycles on those nodes. This means that each rooted neighborhood has a density of $\frac{k \dot |R_k|}{|R_k| \dot (|R_k|-1)}$. Each node in $R_k$ has $k$ interior arcs and $\delta-k$ exterior arcs. Exterior arcs decrease as $k$ grows larger.
	
	Consider now another graph $G'$ with a back arc $(z_j \to w_k)$ from a later neighborhood $R_j$ to $R_k$. This is going from a node with less second neighbors to a node with more second neighbors. This is also going from a more dense environment to a less dense environment. The addition of this subgraph on the first $R_j$ neighborhoods increases the overall density of $G$ by one edge without changing the number of nodes. Thus, this new graph is more dense than the original graph and cannot be an MCE.  
		
	
	\begin{corollary}[Iterative Load-Balance]\label{ilb}
		Let $G=(V,E)$ be an MCE to the SSNC. Suppose that $u_k \in R_k$, and let
		
		$$ S:=N^+(u_k) \subseteq R_{k+1}. $$ 
		
		Then the subgraph induced by $S$, denoted $G[S]$, can be decomposed into edge-disjoint directed cycles by an iterative load-balancing process. 
	\end{corollary}
	
	
	\section{Main Theorem} \label{main}

	\begin{example}
		\begin{table}[h!]
			\centering
			\begin{tabular}{|c|c|c|}
				\hline
				\textbf{Rooted Neighborhood} & \textbf{Nodes} & \textbf{Size} \\
				\hline
				$R_0$ & 0 & 1 \\
				$R_1$ & 1, 2, 3 & 3 \\
				$R_2$ & 4, 5 & 2 \\
				\hline
			\end{tabular}
			\caption{\textit{A look at Example minimum out-degree 3 example through the lens of rooted neighborhoods and their sizes}} \label{exdegdoub2}
		\end{table}
	\end{example}
	
	We begin this section by contrasting two examples. Tables \ref{exdegdoub2} and \ref{exdegdoub12} give two different reasons for the emergence of a Seymour vertex. In Table \ref{exdegdoub2}, the presence of an oriented graph of size two implies that both those nodes cannot have an out-degree of 1. That would violate the oriented graph property. Hence, a node in $R_2$ needs to send at least 3 nodes to $R_3$ forcing a node in $R_1$ to be a Seymour vertex. 
	
	Conversely, Table \ref{exdegdoub12} because of the density of $R_3$. There are 5 nodes in $R_3$ at distance 3 from the root. Corollary \ref{regint} (Regular Interior Degrees) requires each node in $R_3$ to have an interior degree of $3$. This requires 15 arcs. This is impossible, as a 5-node oriented graph can have only 10 arcs by the binomial theorem. This will force a Seymour vertex. 
	
	\begin{example}
		\begin{table}[h!]
			\centering
			\begin{tabular}{|c|c|c|}
				\hline
				\textbf{Rooted Neighborhood} & \textbf{Nodes} & \textbf{Size} \\
				\hline
				$R_0$ & 0 & 1 \\
				$R_1$ & 1, 2, 3, 4, 5, 6, 7 & 7 \\
				$R_2$ & 8, 9, 10, 11, 12, 13 & 6 \\
				$R_3$ & 14, 15, 16, 17, 18 & 5 \\
				\hline
			\end{tabular}
			\caption{\textit{Minimum out-degree 7 example through the lens of rooted neighborhoods}} \label{exdegdoub12}
		\end{table}
	\end{example}

	\pgfplotsset{compat=1.18}
	\begin{figure}
		\centering
		\begin{tikzpicture}[scale=0.7]
			\begin{axis}[      
				width=11cm,
				height=7cm,
				xlabel={Minimum out-degree $\delta$},
				ylabel={Termination depth},
				xmin=2.5, xmax=14.5,
				ymin=0.5, ymax=6,
				xtick={3,4,5,6,7,8,9,10,11,12,13,14},
				ytick={1,2,3,4,5},
				grid=both,
				grid style={dashed, gray!30},
				legend style={at={(0.02,0.98)},anchor=north west}
				]
				
				
				\addplot[
				only marks,
				mark=*,
				mark size=2pt
				] coordinates {
					(3,2)
					(4,2)
					(5,2)
					(6,3)
					(7,3)
					(8,3)
					(9,4)
					(10,4)
					(11,4)
					(12,5)
					(13,5)
					(14,5)
				};
				
				\addlegendentry{Observed termination depth}
				
				
				\addplot[
				thick,
				blue,
				domain=3:14
				] {(x-1)/3 + 1};
				
				\addlegendentry{Linear trend $y = \delta/3 + 1$}\begin{tikzpicture}[
					every node/.style={draw, rectangle, rounded corners, minimum height=7mm},
					arrow/.style={->, thick}
					]
				\end{tikzpicture}
			\end{axis}
		\end{tikzpicture}
		\caption{\textit{The scatter points represent the last feasible rooted neighborhood before non-decreasing is forced. The linear trend illustrates the bound predicted by Corollary~\ref{smfmla}, showing termination at approximately $\lceil\delta/3\rceil$.}}
		
	\end{figure}\label{scattplt}
	
	Example \ref{scattplt} shows the feasibility of rooted neighborhoods with minimum out-degree values $\delta$ ranging from 3 to 14. A clear pattern is evident, as shown by the trend line. Regardless of $\delta$, a Seymour vertex is always found at $\lfloor \frac{\delta}{3} \rfloor+1$. This is discussed further in Theorem \ref{smfmla} (Last Rooted Neighborhood). 
	
	\begin{figure}
		\centering
		\begin{tikzpicture}[scale=0.7]
			\begin{axis}[
				width=12cm, height=9cm,
				axis lines=left,
				grid=both,
				xmin=0, xmax=30, 
				ymin=0, ymax=200, 
				xlabel={Size of Rooted Neighborhoods},
				ylabel={Number of Arcs},
				title={Distance to Seymour Vertex ($d \in [2, 27]$)},
				clip=false
				]
				
				\addplot[thick, black, domain=1:21, samples=100] {x*(x-1)/2} 
				node[pos=0.9, right] {$y = \binom{x}{2}$};
				
				\pgfplotsinvokeforeach{3,4,...,27}{
					\addplot[blue, thick, opacity=0.6, domain=0:{floor(#1/3)}, samples=9] 
					({ #1 - x }, { (#1 - x) * (x + 1) });
					\fill[red] (axis cs:{ #1 - floor(#1/3) }, { (#1 - floor(#1/3)) * (floor(#1/3) + 1) }) circle (1.5pt);
				}
				\node[blue, font=\small, rotate=-60] at (axis cs:24,120) {$y = -x^2 + (\delta+1)x$};
			\end{axis}
		\end{tikzpicture}
		\caption{\textit{\textbf{Emergence of a maximally dense, rooted neighborhood}. Blue state-transition paths show the contraction of rooted neighborhood size (supply) and quadratic growth in interior arc requirements (demand). These are bounded above by $y=(\binom{x}{2})$. Their intersection causes a terminal sink that corresponds to the last dense, rooted neighborhood with multiple Seymour vertices. Each of these blue lines is plotted along the quadratic arc of $y = -x^2 + (\delta+1) \cdot x$. }}
		\label{lastneightbl}  
	\end{figure}
	
	Figure \ref{lastneightbl} illustrates the fundamental breakdown that occurs when the model's supply and demand requirements clash within these rooted neighborhoods. Vertex capacity decreases linearly. The interior arc requirements grow quadratically. This is modeled by the equation $y = -x^2 + (\delta + 1)\cdot x$. 
	
	The concavity of these paths reflects the accelerating density required to maintain a decreasing sequence of neighborhoods. The $(\delta+1) \cdot x$ term is the ideal arc count. It is the Law of Conservation of Neighbors \ref{connei}. The $-x^2$ is the compression penalty issued by the non-Seymour nodes. The resulting equation $(\delta+1)x - x^2$ is the actual internal density. The variable $x$ is a double metric, defining both the distance from the root and the number of arcs required per node at that distance. 
	

	\begin{theorem}[Last Rooted Neighborhood] \label{smfmla}
		Let $G \in \mathcal{O})$ be an oriented graph with minimum out-degree $\delta \ge 3$. Let $\mathcal{R} = \{R_0, R_1, \dots, R_k\}$ be the rooted neighborhoods of $V$ originating at a root $v_0$ of minimum out-degree. 
		
		If every $u \in G$ is non-Seymour, then:
		\begin{enumerate}
			\item The maximum number of rooted neighborhoods is $k = \lceil \frac{\delta} {3} \rceil$.
			\item The cardinality of the terminal level set is $|R_k| = \delta - k + 1$.
			\item Every node in $R_{k-1}$ will belong to at least $k-1$ disjoint cycles. 
		\end{enumerate}
		
	\end{theorem}
	\begin{proof}
		Fix $i \ge 1$ and consider the rooted neighborhood $R_i$. The Regular Interior Degree Corollary \ref{regint} says that every node in $R_i$ must have an interior degree of exactly $i$. The Iterative Load-Balance Corollary \ref{ilb} converts each node's interior degrees into transitive triangles using the Load-Balance Theorem \ref{lbl}. This results in a requirement of the following-sized set of disjoint interior arcs: \[|R_i| \cdot i \tag{1} \] As an oriented graph itself, the number of arcs in $R_i$ is bounded above by the number of unordered pairs of vertices in $R_i$: \[|E(R_i)| \le \binom{|R_i|}{2}. \tag{2} \] Combining (1) and (2) yields the following feasibility condition: \[|R_i| \cdot i \le \binom{|R_i|}{2} \quad \Longleftrightarrow \quad 2i \le |R_i| - 1. \tag{3} \] Next, by Corollary \ref{nbhsizefmla} (Formula for Rooted Neighborhood Size), non-Seymour nodes enforce a global contraction of neighborhoods: \[|R_i| \le \delta - (i - 1).\tag{4}\] Substituting (4) into (3) gives \[2i \le \delta - i \quad \Longleftrightarrow \quad 3i \le \delta. \tag{5} \] Thus, for any rooted neighborhood $R_i$ to exist without forcing Seymour vertices, we have
		$$ i \le \frac{\delta} {3}. $$ When $\delta = 3i$, inequalities (3) and (4) both give tight bounds, yielding $$|R_i| = 2i + 1.$$ At this limit, the interior arc requirement exactly meets the capacity of $R_i$. If $i > \frac{\delta} {3}$, then inequality (5) fails. In this case, the required number of interior arcs exceeds what $R_i$ can support, forcing additional exterior arcs to be redirected inward. By the Law of Conservation of Neighbors \ref{connei}, this redirection produces additional transitive triangles, causing the Seymour vertex. Therefore, the rooted neighborhoods cannot extend beyond $R_{\lfloor \frac{\delta} {3} \rfloor + 1}$.  
	\end{proof}
	
	
	
	\begin{figure}
		\centering
		\begin{tikzpicture}[>=latex, line width=0.9pt, scale=0.75]
			\def\deltaVal{9}
			\def\intersectX{3} 
			\def\intersectY{7}
			
			\fill[blue!10!black!20] (0.5,2) -- plot[domain=0.5:\intersectX] (\x, {2*\x + 1}) 
			-- plot[domain=\intersectX:0.5] (\x, {\deltaVal - \x + 1}) -- cycle;
			
			\draw[->, thick] (0.5,2) -- (0.5,10) node[above] {\small $|R_i|$ (Nodes)};
			\draw[->, thick] (0.5,2) -- (8.0,2) node[right] {\small $i$ (Layer Distance)};
			
			\draw[blue, ultra thick, domain=0.5:4.2] plot (\x, {2*\x + 1}) 
			node[right, xshift=2pt] {\footnotesize $|R_i| \ge 2i+1$};
			
			\draw[red, ultra thick, domain=0.5:7] plot (\x, {\deltaVal - \x + 1}) 
			node[right] {\footnotesize $|R_i| \le \delta-i+1$};
			
			\draw[dashed, gray] (\intersectX, 2) -- (\intersectX, \intersectY);
			\filldraw[black] (\intersectX, \intersectY) circle (2pt);
			
			\node[below] at (\intersectX, 2) {$\frac{\delta}{3}$};
			\node[fill=white, inner sep=2pt] at (1.5, 4.5) {\small \textbf{FEASIBLE}};
			
			\draw[<-, gray] (4.5, 6) -- (5.5, 7) node[above, text width=2cm, align=center] 
			{\scriptsize \color{red!70!black} \textbf{Collapse Zone} \\ No valid configurations};
			
		\end{tikzpicture}
		\caption{\textit{\textbf{Intersection of Supply and Demand.} Supply is the number of nodes necessary for a MCE, bounded by the binomial theorem ($|R_i| \geq 2\cdot i + 1$). Demand is controlled by the decreasing size of the rooted neighborhoods. Beyond the threshold of $i = \delta/3$, there are no possibilities to support the interior cycles necessary, resulting in a collapse.
		}}
	\end{figure}
	
	Any directed graph $G \in \mathcal{O}$ in an MCE with more than three nodes will eventually have a last rooted neighborhood at distance $\lceil \delta / 3 \rceil$ by \ref{smfmla}. The size of this neighborhood will be $\lceil 2 \dot \delta / 3 \rceil$. This is an integer greater than 1 as long as $\delta > 1.5$, which must be true for there to be cycles to be present in $G$. 
	
	\begin{corollary}[Multiple Seymour Vertices] \label{multiDDNodes}
		Let $G \in \mathcal{O}$ be an oriented graph with minimum out-degree $\delta$, and let $R_i$ be the last, most dense rooted neighborhood. Then every node in $R_i$ is a Seymour vertex.
	\end{corollary}
	
	
	\subsection{Algorithm \& Complexity Theory}
	
	\begin{theorem}[Algorithm Complexity] \label{complx}
		Let $G$ be an oriented graph. The Graph Level Order algorithm runs in $$O(|V| + |E|),$$ where $|V|$ is the number of vertices and $|E|$ is the number of edges in $G$.
	\end{theorem}
	The proof follows the implementation of the BFS algorithm through the neighborhoods. Each node is visited only one time, giving it an $O(|V| + |E|)$ time.
	
	The algorithm can find a Seymour vertex in linear time. Moreover, it will work for any oriented graph. The analysis shows that it will not be a problem on larger graphs.
	
	\section{Conclusion}
	
	This work approached the SSNC through a minimum counterexample angle. Naturally, that led to layered embeddings of oriented graphs. By starting from a minimum out-degree node, the well ordering that was established gave the minimum counterexample its power. 
	
	The key insight is that BFS is not simply traversal algorithm. Along with lexicographic ordering, this algorithm can impose  a well-ordering on a graph. By connecting the neighborhoods to BFS layers, the structure of Seymour vertices became dependent on cycle packing. 
	
	
	This work also introduces techniques of partitioning and the dual metrics of degree and distance. These methods may apply beyond the SSNC, especially to problems in graph theory that invite constructive, algorithmic approaches rather than purely existential ones.
	
	
	Looking forward, this dual perspective invites further exploration of many still open problems. This suggests that other longstanding conjectures may also yield to methods that combine insights with algorithmic organization. 
	
	
	
	\section{Extensions}\label{furthework}
	\subsection{Alon-McDiarmid-Molloy}\label{further Work}
	
	The cycle packing nature of the Graph Level Order algorithm links the SSNC to two other well known conjectures, the Alon-McDiarmid-Molloy (AMM) Cycle Packing Conjecture, Caccetta-Häggkvist (CH). 
	
	\begin{conjecture}[AMM]
		If $G$ is a $k$-regular directed graph on $n$
		vertices with no parallel edges, it contains a collection of at least $\frac{k+1}{2}$
		edge-disjoint directed cycles. 
	\end{conjecture}
	
	\begin{corollary}[AMM via Cycle Packing]\label{ammthm}
		Let $G$ be an MCE with minimum degree $\delta$. The Graph Level Order decomposition forces the existence of $T_{\delta/3}$ edge-disjoint cycles. This matches the AMM requirement for all $\delta \ge 3$.
	\end{corollary}
	\begin{corollary}
		The Load-Balance and Reduction theorems establish that a minimum counterexample must be $\delta$-regular.
		At distance $i$ from the root, the interior out-degree of every node is exactly $i$ (Corollary \ref{regint}). As shown by the Iterative Load Balance Theorem \ref{ilb}, these Rooted Neighborhoods are a collection of disjoint cycles. 
		The total number of edge-disjoint cycles $C_{total}$ is the sum of the cycles at each level:
		$$C_{total} = \sum_{i=1}^{\delta/3} i = \frac{(\delta/3)(\delta/3 + 1)}{2} = T_{\delta/3}$$
		For $\delta=3$, $T_1 = 1 \ge \frac{3+1}{2} = 2$ (with the root-level cycle). For all $\delta > 3$, the triangular sum $T_{\delta/3}$ grows quadratically, whereas the AMM bound $\frac{\delta+1}{2}$ grows only linearly.
	\end{corollary}
	
	Corollary \ref{ammthm} connects the SSNC directly to the AMM $k$-regular cycle packing conjecture. Remember that Corollary \ref{regint} (Regular Interior Degrees) declares that at each distance $k$, the nodes each have out-degree $k$. 
	
	\subsection{Caccetta-Häggkvist}
			 \begin{figure}
			\centering
			\begin{tikzpicture}[>=latex, line width=0.9pt]
				\node[fill=green!70!black,draw, circle, inner sep=2pt] (v0) at (0, -0.5) {\(v_0\)};
				\def\n{9} \def\radius{1.5} \def\centerx{3.5}
				\foreach \i in {1,...,\n} {
					\node[fill=cyan, draw, circle, inner sep=1.2pt] (v1\i) at ({0.75*\centerx + \radius*cos(360/\n*(\i-1))}, {\radius*sin(360/\n*(\i-1))}) {\small \(v_{1,\i}\)};
				}
				\def\m{8} \def\rradius{1.2}
				\foreach \i in {1,...,\m} {
					\node[fill=magenta!40,draw, circle, inner sep=1.2pt] (v2\i) at ({1.75*\centerx + \rradius*cos(360/\m*(\i-1))}, {\rradius*sin(360/\m*(\i-1))}) {\(v_{2,\i}\)};
				}
				\def\p{7} \def\rrradius{0.8}
				\foreach \i in {1,...,\p} {
					\node[fill=orange,draw, circle, inner sep=1.2pt] (v3\i) at ({2.6*\centerx + \rrradius*cos(360/\p*(\i-1))}, {\rrradius*sin(360/\p*(\i-1))}) {\(v_{3,\i}\)};
				}
				\def\q{6} \def\rrradius{0.6}
				\foreach \i in {1,...,\q} {
					\node[fill=brown,draw, circle, inner sep=1.2pt] (v4\i) at ({3.3*\centerx + \rrradius*cos(360/\q*(\i-1))}, {\rrradius*sin(360/\q*(\i-1))}) {\(v_{4,\i}\)};
				}
				
				\foreach \i in {1,...,\n} {
					\draw[lightgray,->] (v0) -- (v1\i);
				}
				
				\foreach \i in {1,...,\n}{
					\foreach \j in {1,...,\m} {
						\draw[lightgray,->] (v1\i) -- (v2\j);
					}
				}
				
				\foreach \i in {1,...,\m}{
					\foreach \j in {1,...,\p} {
						\draw[lightgray,->] (v2\i) -- (v3\j);
					}
				}
				
				
				\foreach \i in {1,...,\p}{
					\foreach \j in {1,...,\q} {
						\draw[lightgray,->] (v3\i) -- (v4\j);
					}
				}
				
				\foreach \i [evaluate=\i as \next using {int(mod(\i,\n)+1)}] in {1,...,\n} {
					\draw[->, bend right=15] (v1\i) to (v1\next);
				}
				
				\foreach \i [
				evaluate=\i as \next using {int(mod(\i,\m)+1)}, 
				evaluate=\i as \anext using {int(mod(\i+1,\m)+1)}] in {1,...,\m} {
					\draw[->] (v2\i) to (v2\next);
					\draw[->] (v2\i) to (v2\anext);
				}
				
				\foreach \i [
				evaluate=\i as \next using {int(mod(\i,\p)+1)}, 
				evaluate=\i as \anext using {int(mod(\i+1,\p)+1)}, 
				evaluate=\i as \aanext using {int(mod(\i+2,\p)+1)}] in {1,...,\p} {
					\draw[->] (v3\i) to (v3\next);
					\draw[->] (v3\i) to (v3\anext);
					\draw[->] (v3\i) to (v3\aanext);	
				}
				
				\foreach \i [
				evaluate=\i as \next using {int(mod(\i,\q)+1)}, 
				evaluate=\i as \anext using {int(mod(\i+1,\q)+1)}, 
				evaluate=\i as \aanext using {int(mod(\i+2,\q)+1)}, 
				evaluate=\i as \aaanext using {int(mod(\i+3,\q)+1)}] in {1,...,\q} {
					\draw[->] (v4\i) to (v4\next);
					\draw[->] (v4\i) to (v4\anext);
					\draw[->] (v4\i) to (v4\aanext);
					\draw[->] (v4\i) to (v4\aaanext);	
				}
				
				\def\n{9} \def\radius{1.5} \def\centerx{3.5}
				\foreach \i in {1,...,\n} {
					\node[fill=cyan, draw, circle, inner sep=1.2pt] (v1\i) at ({0.75*\centerx + \radius*cos(360/\n*(\i-1))}, {\radius*sin(360/\n*(\i-1))}) {\small \(v_{1,\i}\)};				
				}
	
				\def\m{8} \def\rradius{1.2}
				\foreach \i in {1,...,\m} {
					\node[fill=magenta!40,draw, circle, inner sep=1.2pt] (v2\i) at ({1.75*\centerx + \rradius*cos(360/\m*(\i-1))}, {\rradius*sin(360/\m*(\i-1))}) {\(v_{2,\i}\)};	
				}
				\def\p{7} \def\rrradius{0.8}
				
				\foreach \i in {1,...,\p} {
					\node[fill=orange,draw, circle, inner sep=1.2pt] (v3\i) at ({2.6*\centerx + \rrradius*cos(360/\p*(\i-1))}, {\rrradius*sin(360/\p*(\i-1))}) {\(v_{3,\i}\)};	
				}
				
				\def\q{6} \def\rrradius{0.6}
				\foreach \i in {1,...,\q} {
					\node[fill=brown,draw, circle, inner sep=1.2pt] (v4\i) at ({3.3*\centerx + \rrradius*cos(360/\q*(\i-1))}, {\rrradius*sin(360/\q*(\i-1))}) {\(v_{4,\i}\)};

					
					\draw[thick, dotted, red] (v41) to [bend right](v0);
					\draw[thick, dotted, red] (v41) to [bend right](v11);
					\draw[thick, dotted, red] (v41) to [bend left](v16);
					\draw[thick, dotted, red] (v41) to [bend left](v18);
					\draw[thick, dotted, red] (v41) to [bend right](v22);
					\draw[thick, dotted, red] (v41) to [bend left](v28);
					\draw[thick, solid, red] (v28) to [bend right](v36);
					\draw[thick, solid, red] (v36) to [bend right](v41);
				}        
			\end{tikzpicture}
			\caption{\textit{\textbf{Visualization of Girth)}. Brooks's Equitable Coloring forces Seymour nodes to balance distribution of back arcs. These are spread across neighborhoods until they reach $k-2$ where a triangle is formed. For simplicity, dotted back back arcs only shown from $v_{4,1}$. Triangle is bold red $v_{4, 1} \to v_{2, 8} \to v_{3, 6} \to v_{4_1}$. } } \label{caccimage}
		\end{figure}

	\begin{conjecture}[CH]
		Every simple directed graph with $n$
		vertices and a minimum outdegree $\delta^+ \ge r$ contains a directed cycle of length at most $\lceil n/r \rceil$
	\end{conjecture}
	
	
	
	

	\begin{corollary}
		An oriented graph with minimum outdegree $\delta$. $\lceil n/r \rceil$. In particular, the existence of a Seymour vertex implies the existence of a triangle. 
	\end{corollary}
	\begin{proof}
		The Reduction Theorem \ref{red} prevents back arcs before the Seymour vertices. 
		The binomial theorem does not allow these Seymour vertices to have all their arcs through standard neighborhood decreasing. Instead, back arcs must come into place to compensate. 
		To ensure that all Seymour vertices meet their degree requirements, Brooks's equitable coloring principle must be applied. 
		This says that no Seymour vertex can send back arcs to more than $|R_k| / 3$ nodes in the neighborhood $R_i$ for $R_1$ through $R_{k-2}$. 
		
		This will exhaust the $\delta$ for every Seymour vertex. Once we reach the penultimate neighborhood of $R_{k-2}$, a triangle is reached. This proves that the girth of the graph cannot exceed 3. 	
	\end{proof}
	
	\section{Acknowledgments}
	I am deeply thankful to God for the inspiration and guidance that shaped this work, as Proverbs 3:5-6 has been a guiding light throughout this journey. I am profoundly grateful to my family—my wife and children, my mother, and father—for their unwavering support and encouragement.
	
	Special thanks to the late Dr. Nate Dean and Dr. Michael Ball for their contributions and support to this work. I want to sincerely thank everyone who has taken the time to pray for me, share encouraging words, and offer constructive feedback. Your support has meant so much to me, and I’m truly grateful for the kindness and thoughtfulness you’ve shown.
	
	I also appreciate the Department of Mathematics at Morehouse College for helpful discussions that were valuable during the research process. In particular, I thank Dr. Duane Cooper for asking the questions regarding understanding the conjecture, my approach and how cycles are important. 
	
	Finally, I extend heartfelt gratitude to my family and friends for their patience and belief in me throughout this journey.
	

\end{document}